%% file: multiquandles.tex
\documentclass[reqno, a4paper, 12pt]{amsart}
\usepackage{amsmath,amsfonts}
  \usepackage{paralist}
  \usepackage{graphics} %% add this and next lines if pictures should be in esp format
  \usepackage{epsfig} %For pictures: screened artwork should be set up with an 85 or 100 line screen
  \usepackage{tikz-cd}
  \usepackage{tikz}
\usetikzlibrary{positioning, arrows.meta}

  \usepackage{xfrac}
  \usepackage{enumitem}
  \usepackage[justification=centering]{caption}
\usepackage{graphicx}  \usepackage{epstopdf}%This is to transfer .eps figure to .pdf figure; please compile your paper using PDFLeTex or PDFTeXify.
 \usepackage[colorlinks=true]{hyperref}
\hypersetup{urlcolor=blue, citecolor=red}

\newtheorem{theorem}{Theorem}[section]
\newtheorem{corollary}[theorem]{Corollary}

\newtheorem{lemma}[theorem]{Lemma}

\theoremstyle{definition}
\newtheorem{definition}[theorem]{Definition}
\newtheorem{example}[theorem]{Example}
\newtheorem{remark}[theorem]{Remark}

\begin{document}

	\title{colored link Invariants}
\author{Georgy C Luke}
	\address{Indian Institute of Science Education and Research Tirupati,
	Tirupati -517507. Andhra Pradesh, INDIA.}
\email{georgy.c@students.iisertirupati.ac.in}

\author{B. Subhash}
\address{Indian Institute of Science Education and Research Tirupati,
	Tirupati -517507. Andhra Pradesh, INDIA.}
\email{subhash@iisertirupati.ac.in}
%\thanks{The second author would like to thank DST-SERB for the early career research grant [ECR/2017/003417]}



\maketitle

% Enter the first author's name and address:
%\centerline{\scshape Georgy C Luke$^*$}
%\medskip
%{\footnotesize
% please put the address of the first author
 %\centerline{IISER Tirupati,}
 %  \centerline{Tirupati}
%   \centerline{ Andhrapradesh, India - 517507}
%} % Do not forget to end the {\footnotesize by the sign }

%\medskip

%\centerline{\scshape B. Subhash}
%\medskip
%{\footnotesize
 % please put the address of the second  and third author
% \centerline{ IISER Tirupati,}
%   \centerline{Tirupati}
%   \centerline{Andhrapradesh, India - 517507}
%}

%\bigskip

% The name of the associate editor will be entered by an editorial staff
% "Communicated by the associate editor name" is not needed for special issue.
% \centerline{(Communicated by the associate editor name)}

%The abstract of your paper
\begin{abstract}
This article presents new colored link invariants by introducing the concepts of multi-quandles and topological multi-quandles.
\end{abstract}

%The title of your section 1
\section{Introduction}
An oriented colored link has at least two components, each labeled by one from a set of colors. Two such colored links are isotopic if there is an orientation preserving ambient isotopy between them, which preserves colors. The multivariable Alexander polynomial is the first known colored link invariant. In 1970, Conway introduced the potential function of a colored link determined by the Alexander polynomial, which has properties the latter doesn't have \cite{hartley1983conway}.  For a link whose components are labeled from $1$ to $n$, both the invariants give an integral $L-$polynomial in variables $t_{1},t_{2},\ldots, t_{n}$. In \cite{akutsu1992invariants}, multivariable isotopy invariants of colored oriented links from $N$-state colored braid matrices are constructed. These are generalizations of the multivariable Alexander polynomial. The colored Jones polynomial is another colored link invariant for which each of the components of the colored link is labeled with irreducible representations of the Lie algebra $sl(2,\mathbb{C}).$
In 1990, Jim Hoste and Mark E. Kidwell studied dichromatic links and introduced a chromatic skein invariant for 1-trivial dichromatic links, a particular class of dichromatic links \cite{hoste1990dichromatic}. In \cite{bataineh2020involutory} and \cite{lee2021diquandles}, the notion of dikeis and diquandles was introduced and shown that cardinality of colorings of colored links by dikeis and diquandles are dichromatic link invariants. %Topological quandles was introduced by Rubinzstein in \cite{rubinsztein2007topological}. He used topological quandles to associate a topological space to links. % After that quandles have been studied extensively by topologists for the construction of various knot invariants \cite{carter2003quandle} and algebraists for the study of Hopf algebras \cite{andruskiewitsch2003racks} and Yang Baxter equations.
 
 % \par Colored links are oriented links whose components are colored. Multivariable Alexander polynomial and Conway's potential function are some invariants for colored links\cite{hartley1983conway}.% Some invariants for such links are studied in \cite{hartley1983conway}, \cite{akutsu1992invariants} and \cite{deguchi1994multivariable}. 
\par The purpose of this paper is to introduce new colored link invariants. 
 In the second section, we introduce the idea of multi-quandles and show that the cardinality of colorings of colored links by multi-quandles is a colored link invariant. The third section of the article deals with the theory of colored braids and colored Markov moves. In the last section, topological multi-quandles are defined and used to construct another colored link invariant.
\section{Basic Definitions and Examples}
%\subsection{Basic Definitions}
 \begin{definition}
\par A $quandle$ X is a set with a binary operation $\triangleright$ on X, which satisfies the following conditions.
\begin{enumerate}
 \item $x \triangleright x = x$  $\forall$ $x \in X$
 \item The map $\beta_{y} :X\rightarrow X$ defined by $\beta_{y}(x) = x \triangleright y $  for every $y \in X$ is invertible.
 \item $(x \triangleright y)  \triangleright z = (x \triangleright z ) \triangleright (y \triangleright z )  $
\end{enumerate}
 \end{definition}
 
 See \cite{elhamdadi2015quandles} for more details on quandles.
  \begin{definition}
A $topological$ $quandle$ X is a topological space with a quandle structure such that the mapping 
     \begin{align*}
  \triangleright :X \times X \rightarrow X     
\end{align*}       
 is continuous and $\beta_{y}$ is a homeomorphism for every y $\in$ X \cite{elhamdadi2016foundations}.
  \end{definition}
   \begin{example}
The space $S^{n}\subset \mathbb{R}^{n+1}$ is a topological quandle under the operation $x\triangleright y=2 \langle x,y\rangle y-x$.
 \end{example}
 \begin{definition}
A $dichromatic$ $link$ is a link whose each component is labeled by one of two colors, and both colors are used to label all the link components. Usually, the colors are represented by the labels $1$ and $2.$
   \end{definition}
   A dichromatic link is represented by a $dichromatic$ $link$ $diagram$, a classical link diagram with labels $1$ and $2$ on the components indicating the color of components of the link.
Two dichromatic links are equivalent if they are ambient isotopic and if the isotopy preserves the orientation and colors of the components \cite{lee2021diquandles}.
    \begin{definition}
   Let $\left(X,\triangleright_{1} \right) $ and $\left(X,\triangleright_{2} \right) $ be two quandles whose underlying sets are the same. Then $\left(X,\triangleright_{1},\triangleright_{2} \right) $ is called a $diquandle$ if 
   $$ (x\triangleright_{i}y)\triangleright_{j}z=(x\triangleright_{j}z)\triangleright_{i}(y\triangleright_{j}z) $$ for every $x,y,z \in X$ and $i,j\in \lbrace 1,2\rbrace.$ 
 \end{definition}
 \begin{example}
    Let  $A=Z[t_{1}^{\pm 1},t_{2}^{\pm 1}]$ and $X$ be an $A-$module. Then $(X,\triangleright_{1},\triangleright_{2})$ forms a diquandle under the quandle operations
    $$x\triangleright_{i}y=t_{i}x+(1-t_{i})y$$ for $i \in \lbrace 1,2 \rbrace.$
 \end{example}
 \begin{definition}
     A $diquandle$ $coloring$ of a dichromatic diagram $D$ by a diquandle $X$ is a function that maps the set of arcs of $D$ to X such that the relations depicted in figure \ref{figure:crossing} holds at positive and negative crossings. The set of all diquandle colorings of $D$ by $X$ is denoted by $Col_{X}(D).$
 \end{definition}
 The following result about dichromatic links is from \cite{lee2021diquandles}.
 \begin{theorem}
    Let $L$ be a dichromatic link, $D$ be it's diagram and $X$ be a diquandle. Then, the cardinality of $Col_{X}(D)$ is a dichromatic link invariant.
 
 \end{theorem}
% \subsection{Examples}
   
%The title of your section 2
\section{colored links}
\begin{definition}
A $k-colored$ $link$ is a link whose each component is colored by one of the $k$ colors, and all of the $k$ colors are used to color the link \cite{cimasoni2004geometric}.  
\end{definition}
A $k-colored$ $link$ can be represented by a $k-colored$ $link$ $diagram$ which is a classical link diagram with color labelings of the components denoted by entries from the set $\lbrace 1,2,\ldots,k \rbrace$. Two $k-colored$ $links$ $L$ and $L'$ are said to be $equivalent$ if there exists an ambient isotopy $h:\mathbb{R}^{3}\times \left[0,1 \right] \rightarrow \mathbb{R}^{3} $ preserving the component labels and orientation. Similar to classical knot theory, we could talk about reidemeister moves but colored ones as in figure \ref{figure:R1&R2} and figure \ref{figure:R3}.
\begin{figure}[h]
% \centering \includegraphics[scale=.75]{Capture2.PNG}
\def\svgwidth{300px}
\centering 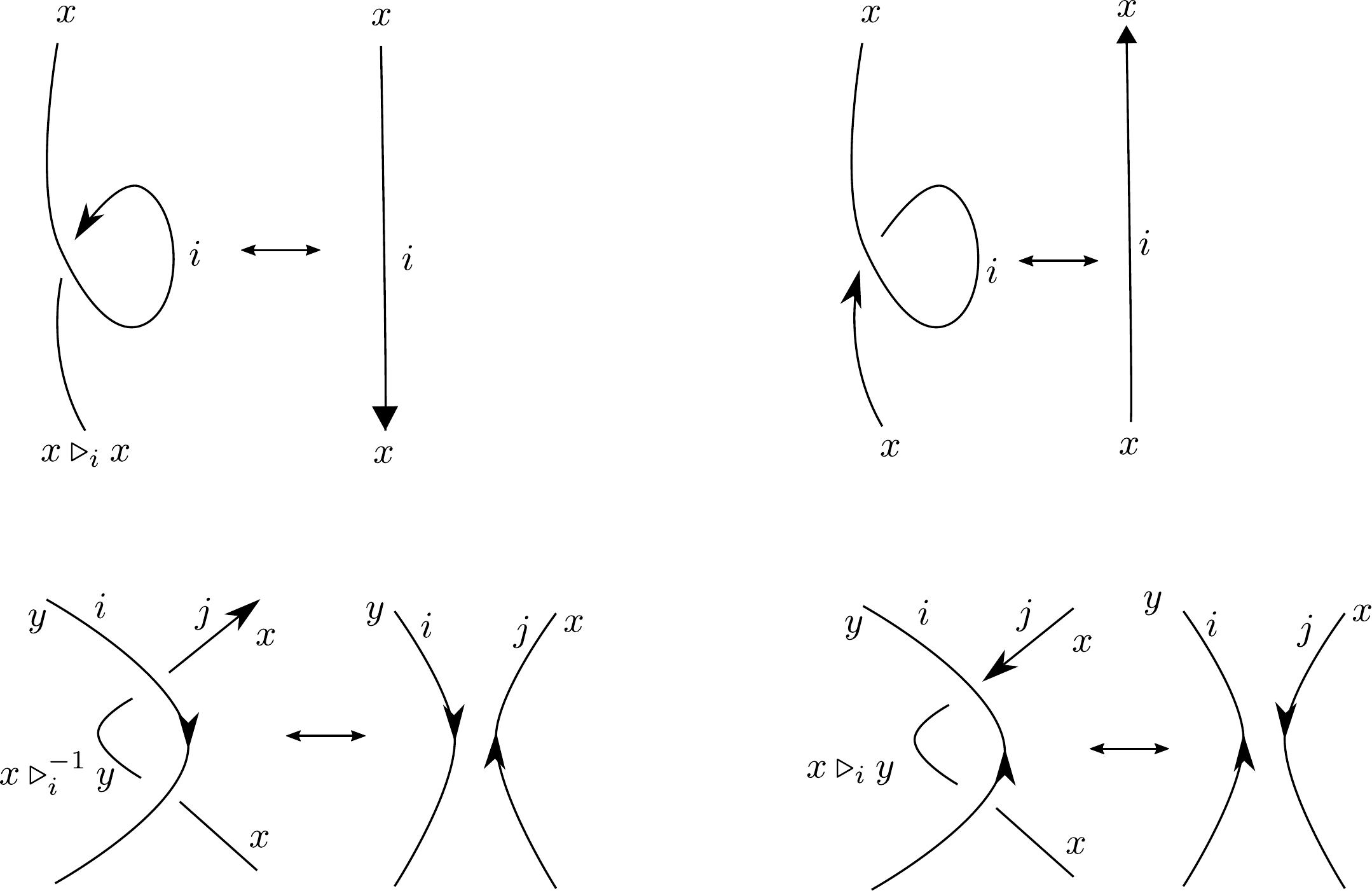
 \caption{Colored reidemeister 1 and 2 moves}
 \label{figure:R1&R2}
 \end{figure}
 \begin{figure}[h]
% \centering \includegraphics[scale=.75]{Capture2.PNG}
\def\svgwidth{300px}
\centering 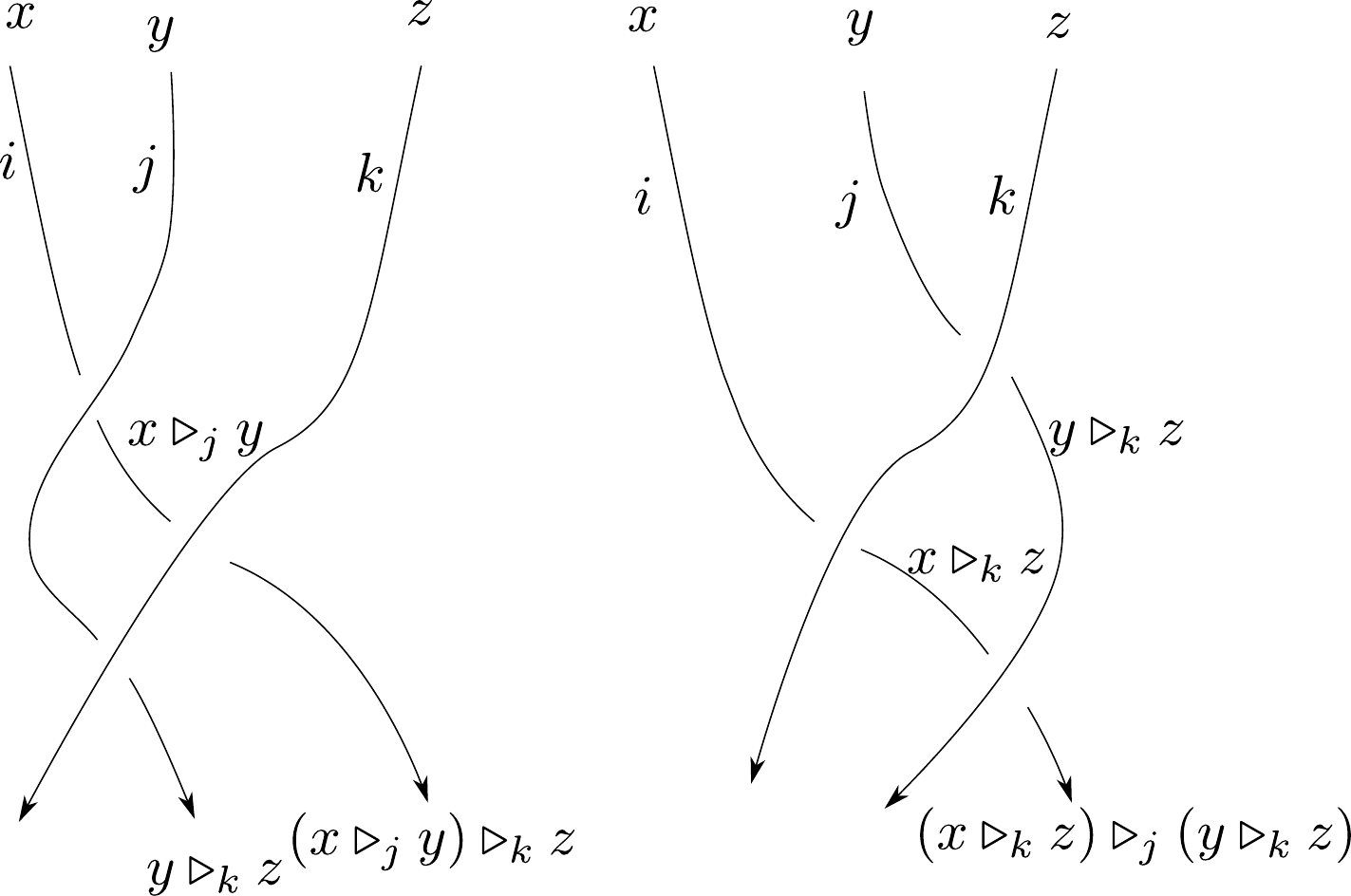
 \caption{Colored third Reidemeister move}
 \label{figure:R3}
 \end{figure}
\par The following theorem is an adaptation of the classical Reidemeister theorem in the context of colored links. A similar theorem for dichromatic links is given in \cite{lee2021diquandles}.
\begin{theorem}
Two $k$-colored links $L$ and $L'$ are equivalent if and only if their corresponding diagrams are related by colored
Reidemeister moves.
\end{theorem}
\begin{definition}
\label{defn:k-quandle}
A $k-quandle$ is a set $X$ equipped with $k$ quandle operations, $(\triangleright_{i})_{i=1}^{k}$ satisfying the following conditions.
\begin{enumerate}
\item $(X,\triangleright_{i})$ for $i \in \lbrace 1,\ldots ,k \rbrace$ forms a quandle.
\item  $(x \triangleright_{i} y)  \triangleright_{j} z = (x \triangleright_{j} z ) \triangleright_{i} (y \triangleright_{j} z )  $ for every $i,j\in \lbrace1,\ldots,k \rbrace$ and for every $x,y,z \in X$.
\end{enumerate}
\end{definition}
We call such sets multi-quandles in general without mentioning the number of underlying quandle operations. Some properties of multi-quandles are illustrated in the following theorems, which are analogous to the case of quandles and diquandles.
\begin{theorem}
    Let $(X, \triangleright_m)_{m=1}^k$ be a $k$-quandle. Let $x,y,z \in X$ and $i,j \in \lbrace 1,2,\ldots,k \rbrace $. Then, the following statements hold.
    \begin{enumerate}
        \item  $(x \triangleright_{i} y)  \triangleright_{j}^{-1} z = (x \triangleright_{j}^{-1} z ) \triangleright_{i} (y \triangleright_{j}^{-1} z ) $
        
        \item  $(x \triangleright_{i}^{-1} y)  \triangleright_{j} z = (x \triangleright_{j} z ) \triangleright_{i}^{-1} (y \triangleright_{j} z ) $
        
        \item  $(x \triangleright_{i}^{-1} y)  \triangleright_{j}^{-1} z = (x \triangleright_{j}^{-1} z ) \triangleright_{i}^{-1} (y \triangleright_{j}^{-1} z ) $
    \end{enumerate}
\end{theorem}
\begin{theorem}
    Let $(X, \triangleright_m)$ for $m=1,\ldots,k$ be $k$ number of quandles. Let $x,y,z \in X$ and $i,j \in \lbrace 1,2,\ldots,k \rbrace $. If
       $$(x \triangleright_{i}^{-1} y)  \triangleright_{j}^{-1} z = (x \triangleright_{j}^{-1} z ) \triangleright_{i}^{-1} (y \triangleright_{j}^{-1} z ) $$
    then the following statements hold.
    \begin{enumerate}
        \item  $(x \triangleright_{i} y)  \triangleright_{j}^{-1} z = (x \triangleright_{j}^{-1} z ) \triangleright_{i} (y \triangleright_{j}^{-1} z ) $
        
        \item  $(x \triangleright_{i}^{-1} y)  \triangleright_{j} z = (x \triangleright_{j} z ) \triangleright_{i}^{-1} (y \triangleright_{j} z ) $
        
        \item  $(x \triangleright_{i} y)  \triangleright_{j} z = (x \triangleright_{j} z ) \triangleright_{i} (y \triangleright_{j} z ) $
    \end{enumerate}
\end{theorem}
The proofs of the above theorems are similar to analogous theorems in \cite{lee2021diquandles}.
\subsection*{Coloring of k-colored links by k-quandles}

 Let $X$ be a finite $k$-quandle. Then, the coloring of a $k$-colored link diagram is a function from the set of arcs of the diagram to $X$ such that it satisfies the crossing relations at a positive and negative crossing as in figure \ref{figure:crossing}. That is if the coloring maps the over arc to 
 $y$, the incoming arc to $x$ and the outgoing arc to $z$ at a crossing of a $k$-colored link diagram $D$, then $z$ should be equated to $x \triangleright_{j} y$ at a positive crossing and to $x \triangleright_{j}^{-1} y$ at a negative crossing, where $j$ indicates the color of the over-strand. The 
  \begin{figure}[h]
% \centering \includegraphics[scale=.75]{Capture2.PNG}
\def\svgwidth{200px}
\centering 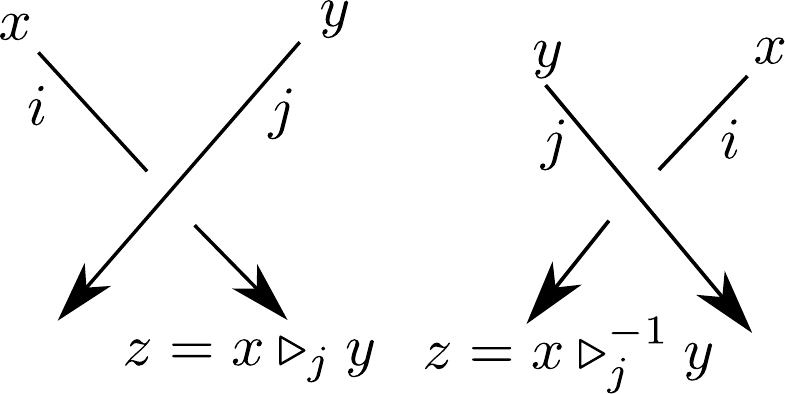
 \caption{Crossing relations}
 \label{figure:crossing}
 \end{figure}
set of colorings of a $k$-colored link diagram $D$ by a $k$-quandle $X$ is denoted by $Col_{X}(D).$  
 \begin{theorem}
Let $L$ be a $k$-colored link and $D$ be its diagram. 
Suppose $X$ is a $k$-quandle. Then, the cardinality of $Col_{X}(D)$ is a colored link invariant.
\end{theorem}
\begin{proof}
See figures \ref{figure:R1&R2} and \ref{figure:R3}. For a diagram $D$, there exists a unique coloring by $X$ before and after the colored reidemeister moves. So the cardinality of $Col_{X}(D)$ remains unchanged before and after the colored reidemeister moves.
\end{proof}
 \begin{example}
 Let $X=\lbrace1,2,3,4,5\rbrace$ be a $3-quandle$ with presentation matrix as follows.
 \begin{center}
$M_{X}= \left[\begin{array}{c c c c c | c c c c c | c c c c c}
 1&4&5&5&4&1&1&1&5&4&1&4&5&1&1\\
 3&2&2&3&3&3&2&2&3&3&2&2&2&2&2\\
 2&3&3&2&2&2&3&3&2&2&3&3&3&3&3\\
 5&5&1&4&1&5&4&4&4&1&4&5&1&4&4\\
 4&1&4&1&5&4&5&5&1&5&5&1&4&5&5
 \end{array}\right]$
 \end{center}
 \end{example}  
 The first, second, and third block of the matrix corresponds to the operation table of $\triangleright_{1}, \triangleright_{2}$ and $\triangleright_{3}$, respectively. The inverse of $M_{X}$, $M_{X}^{-1}$ is given below.
  \begin{center}
$M_{X}^{-1}= \left[\begin{array}{c c c c c | c c c c c | c c c c c}
 1&5&4&5&4&1&1&1&5&4&1&5&4&1&1\\
 3&2&2&3&3&3&2&2&3&3&2&2&2&2&2\\
 2&3&3&2&2&2&3&3&2&2&3&3&3&3&3\\
 5&1&5&4&1&5&4&4&4&1&4&1&5&4&4\\
 4&4&1&1&5&4&5&5&1&5&5&4&1&5&5
 \end{array}\right]$
 \end{center}
 The $i^{th}$ block corresponds to the operation table of $\triangleright_{i}^{-1}$ for $i=1,2,3$. In figure \ref{figure:tricolored links}, you can see the tricolored links, which include the colored Borromean ring $\left( 6^{3}_{2}\right) $ and two links with different colorings but the same underlying link structure $\overline{\sigma_{1}^{-2}\sigma_{2}^{2}}$.

 \begin{figure}[h]
% \centering \includegraphics[scale=.75]{Capture2.PNG}
\def\svgwidth{200px}
\centering 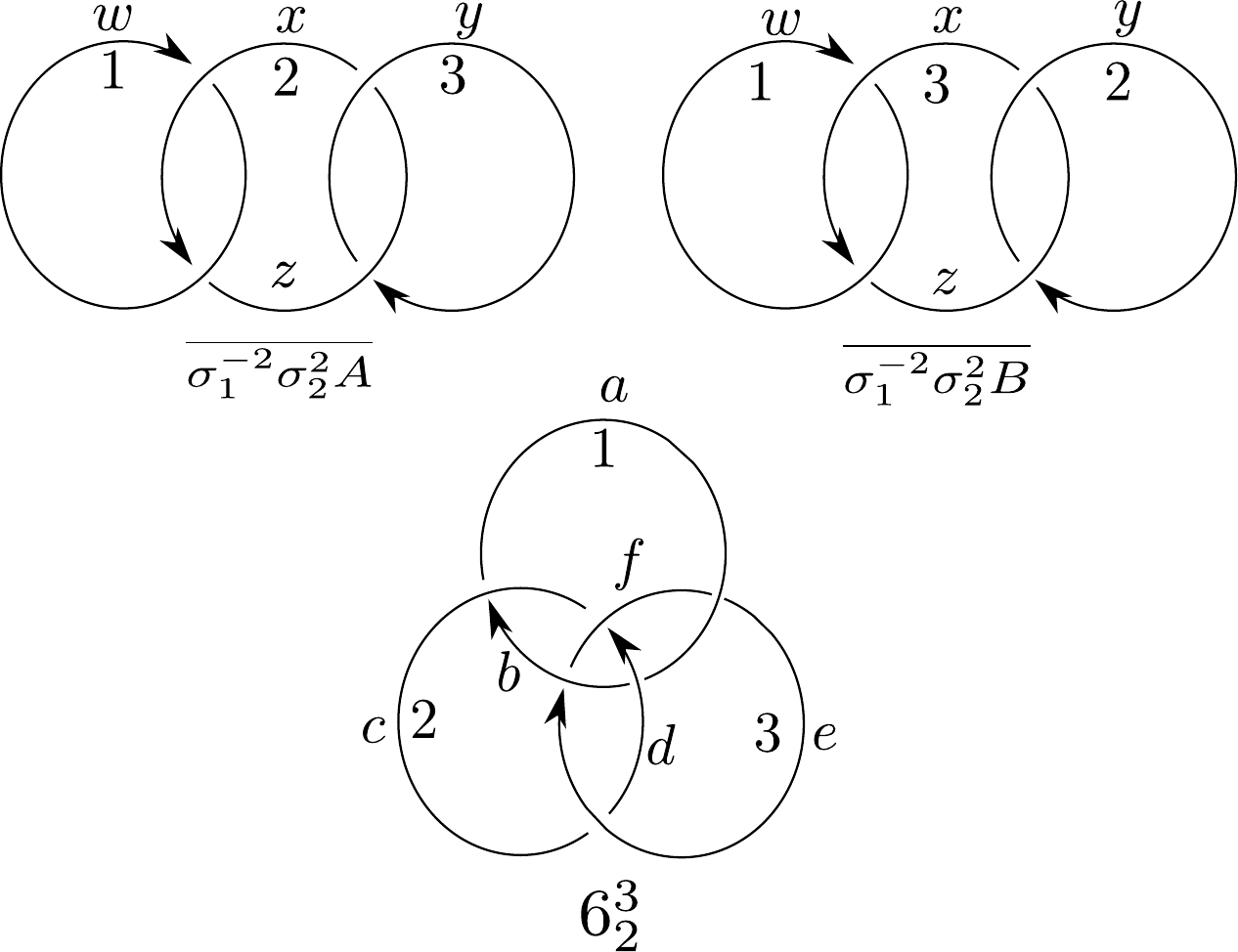
 \caption{Tricolored links}
 \label{figure:tricolored links}
 \end{figure}
 The coloring of the arcs of the links in figure \ref{figure:tricolored links} are given by the following equations.
 \begin{equation*}
  \begin{split}
  Col_{X}\left( \overline{\sigma_{1}^{-2}\sigma_{2}^{2}}A\right)=\lbrace \left(w,x,y,z \right)\in X\mid w=w \triangleright_{2} x,z=x\triangleright_{1}w,\\x=z\triangleright_{3}y,y=y\triangleright_{2}z\rbrace \\
  Col_{X}\left( \overline{\sigma_{1}^{-2}\sigma_{2}^{2}}B\right)=\lbrace \left(w,x,y,z \right)\in X\mid w=w \triangleright_{3} x,z=x\triangleright_{1}w,\\
  x=z\triangleright_{2}y,y=y\triangleright_{3}z\rbrace \\  
  Col_{X}\left(6 ^{3}_{2}\right)= \lbrace \left(a,b,c,d,e,f \right)\in X\mid b=a \triangleright_{2}c,b=a\triangleright_{2}d,\\c=d\triangleright_{3}f, 
   c=d\triangleright_{3}e,e=f\triangleright_{1}a,e=f\triangleright_{1}b \rbrace\\
\end{split}
\end{equation*}
Then \[
\left| \text{Col}_X \left( \overline{\sigma_1^{-2}\sigma_2^2 A} \right) \right| = 23, \,
\left| \text{Col}_X \left( \overline{\sigma_1^{-2}\sigma_2^2 B} \right) \right| = 29 \, \text{and}
\left| \text{Col}_X \left( 6^3_2 \right) \right| = 71.
\]

Note that $\overline{\sigma_1^{-2}\sigma_2^2 A}$ and $\overline{\sigma_1^{-2}\sigma_2^2 B}$ are distinguished by the cardinality of colorings by $X$ even though both the colored links have the same underlying link structure.

\begin{example}
    Consider the links given in figure \ref{fig:L9n27} and \ref{fig:L10n107}. The multivariable Alexander polynomial evaluates to zero for both links \cite{knotatlas}. We consider all the $3$-colored links with underlying link structures of $L9n27$ and $L10n107.$ We list the cardinality of colorings of the aforementioned colored links by $X$ in the following table. The colored links with the link structure $L9n27$ and $L10n107$ are represented by the tuples $(i,j,k)$ and $(i,j,k,l)$ respectively, where each entry of the tuples indicates the color of the component which is labeled by same entry in figure \ref{fig:L9n27} and \ref{fig:L10n107}. 
    
    \iffalse
    \begin{figure}[h]
% \centering \includegraphics[scale=.75]{Capture2.PNG}
\def\svgwidth{100px}
\centering \input{fig11.pdf_tex}
 \caption{L9n27}
 \label{figure:L9n27}
 \end{figure}
  \begin{figure}[h]
% \centering \includegraphics[scale=.75]{Capture2.PNG}
\def\svgwidth{100px}
\centering \input{fig12.pdf_tex}
 \caption{L10n107}
 \label{figure:L10n107}
 \end{figure}
 \fi
 \begin{figure}
\centering
\begin{minipage}{.5\textwidth}
  \centering
  \includegraphics[width=.4\linewidth]{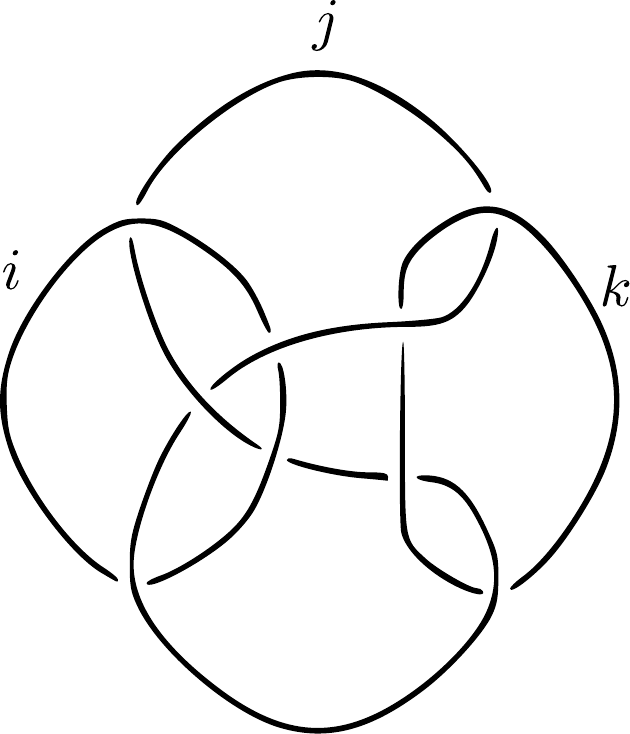}
  \captionof{figure}{L9n27}
  \label{fig:L9n27}
\end{minipage}%
\begin{minipage}{.5\textwidth}
  \centering
  \includegraphics[width=.4\linewidth]{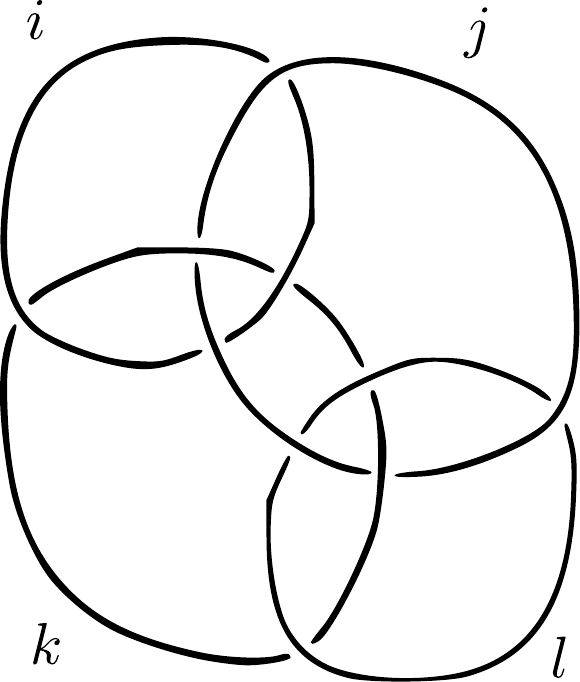}
  \captionof{figure}{L10n107}
  \label{fig:L10n107}
\end{minipage}
\end{figure}
\begin{center}
\label{L9n27}
\begin{tabular}{|c|c|c|c|c|c|}
\hline 
\multicolumn{6}{|c|}{Colored links with underlying link structure L10n107} \\ 
\hline 
D & $(1,1,2,3)$ & $(1,2,2,3)$ & $(1,3,2,3)$ & $(1,1,3,2)$ & $(1,3,3,2)$  \\ 
\hline 
$\vert Col_{X}(D)\vert$ & 307 & 433 & 337 & 265 & 481\\ 
\hline 
D & $(1,2,3,2)$ & $(2,2,1,3)$ &$(2,1,1,3)$  & $(2,3,1,3)$ & $(3,3,2,1)$ \\ 
\hline 
$\vert Col_{X}(D)\vert$ & 337 & 307 & 181 & 283 & 337 \\
\hline
D & $(3,2,2,1)$ & $(3,1,2,1)$ &$(3,3,1,2)$ & $(3,1,1,2)$ & $(3,2,1,2)$ \\ 
\hline
 $\vert Col_{X}(D)\vert$& 433 & 307& 283 & 181 & 307 \\
 \hline
 D&(2,2,3,1)&(2,3,3,1)&\multicolumn{3}{c|}{(2,1,3,1)}\\
 \hline
 $\vert Col_{X}(D)\vert$& 337& 481& \multicolumn{3}{c|}{265}\\
 \hline
\end{tabular} 
\end{center}
\begin{center}
\label{table:L10n107}
\begin{tabular}{|c|c|c|c|c|c|c|}
\hline 
\multicolumn{7}{|c|}{Colored links with underlying link structure L9n27} \\ 
\hline 
D & $(1,2,3)$ & $(1,3,2)$ & $(2,1,3)$ & $(3,2,1)$ & $(3,1,2)$ & $(2,3,1)$ \\ 
\hline 
$\vert Col_{X}(D) \vert$ & 77 & 125 & 95 & 77 & 95 & 125 \\ 
\hline 
\end{tabular} 
\end{center}
\end{example}
   \section{Multicolored braids}  
This section is motivated from \cite{kaul1994chern}.
 \subsection*{Multicolored braids}
 \begin{definition}
 A $k-colored$ braid is an $n$-braid whose each strand is colored by one of $k$ colors, and all the $k$ colors are used for coloring the braid.
 \end{definition}
  Throughout the rest of the article, we assume the braids are oriented from top to bottom unless otherwise mentioned. The colors of the braid are denoted by elements of the set $\lbrace1,2,\ldots,k\rbrace$. Two colored braids $b$ and $b'$ are said to be $isotopic$ if they are isotopic in the sense of braids by a map $F:b \times I \rightarrow \mathbb{R}^{2} \times I$ and $F$ preserve the colorings and orientations of $b$ and $b'.$ The $colored$ $braid$ $diagram$ of a colored braid is the braid diagram with labels of the colorings of the strings.  Two colored braid diagrams are said to be $diagrammatically$ $isotopic$ if there is an isotopy of underlying braid diagrams, which preserves the colorings and orientations of the colored braid diagrams. 
 The colored braid Reidemeister moves are the Reidemeister 2 and 3 moves given in figure \ref{figure:braid moves}.
 \begin{figure}[h]
% \centering \includegraphics[scale=.75]{Capture2.PNG}
\def\svgwidth{200px}
\centering 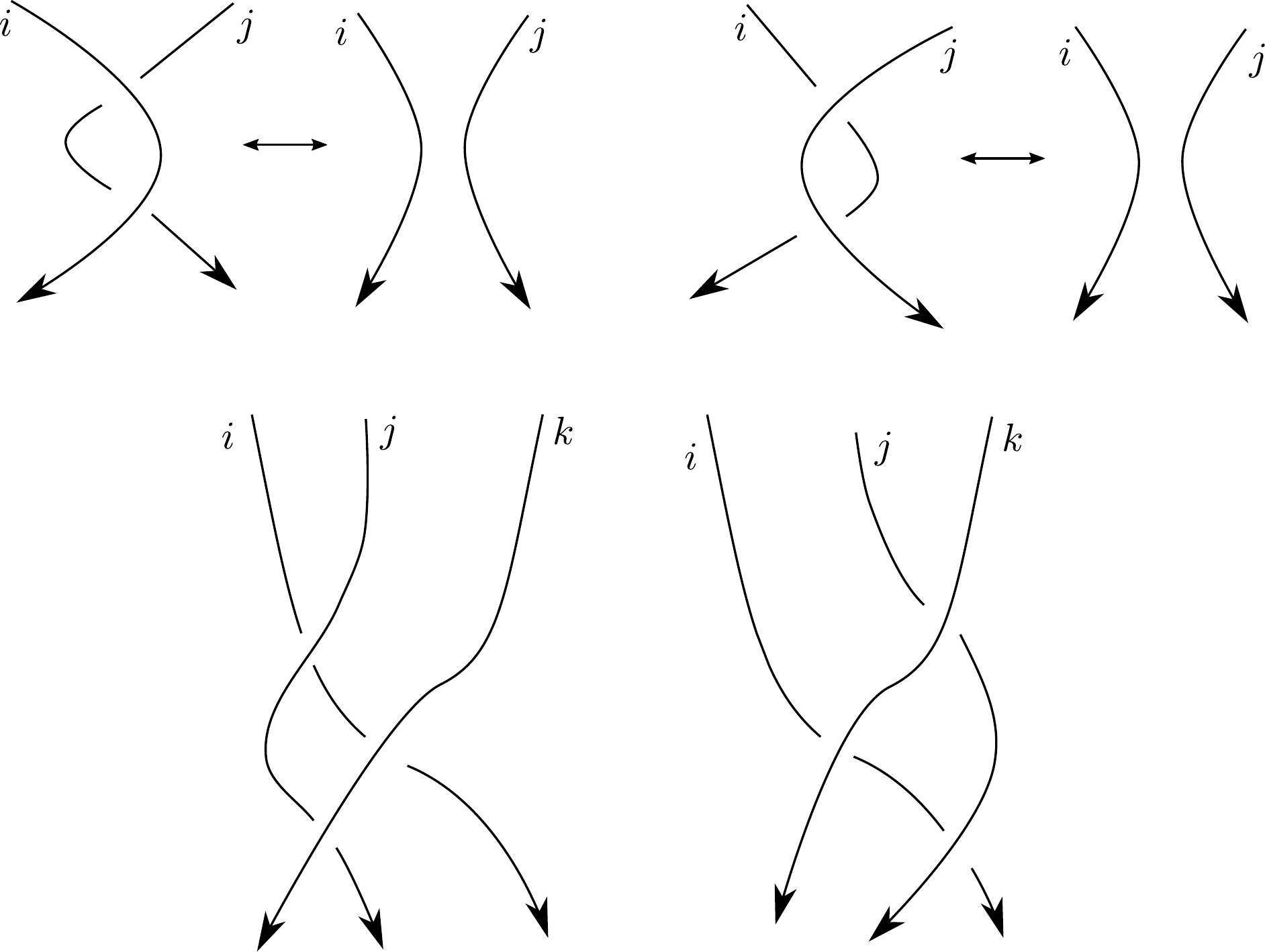
 \caption{Colored braid Reidemeister moves}
 \label{figure:braid moves}
 \end{figure}
 
 Two colored braid diagrams $D$ and $D'$ are said to be $R-equivalent$ if they are related by colored braid Reidemeister moves and diagram isotopy.\\
 The category $\mathcal{B}_{n}^{k}$ of $k$-colored $n$-braids is defined as follows: The objects are $n$-tuples $v=(v_{1},v_{2},\ldots,v_{n})$ with $v_{i}\in \lbrace 1,2,\ldots,k \rbrace$ and the set of morphisms $\textbf{Mor}(v,w)$  are the set of colored braids whose tuple of colors at top and bottom ordered from left to right are $v$ and $w$ respectively. The composition map $\textbf{Mor}(u,v) \times \textbf{Mor}(v,w) \rightarrow \textbf{Mor}(u,w)$ is the colored braid obtained by stacking the second braid below the first one. The composition is associative, and the trivial braid with coloring $v$ acts as an identity morphism of $v.$
 
 \begin{theorem}
 %Two $k$-colored braid diagrams represent isotopic geometric braids if and only if the corresponding diagrams are $R$-equivalent.
 Two $k$-colored braid diagrams represent isotopic braids if and only if the corresponding diagrams are $R$-equivalent.
 \end{theorem}
 \begin{proof}
 The proof is similar to the analogous theorem of braids from section $1.2.3$ of \cite{kassel2008braid}, except in this case, the braids are colored.
 \end{proof}
  We denote a $k$-colored $n$-braid diagram by ${}_w^v\sigma$, where the color of the strands at the top and bottom ends from left to right of the braid are denoted by $v,w\in\lbrace 1,2,\ldots,k \rbrace^{n}$ respectively. The $k$-colored elementary $n$-braid diagrams for ${}^v_w\sigma_{i}$ and ${}^v_w\sigma_{i}^{-1}$ for $i \in \lbrace 1,\ldots,n-1 \rbrace$ are illustrated in figure \ref{figure:elementary braids }, where $v=\left( v_{m}\right)_{m=1}^{n} $ and $w=\left( w_{m}\right)_{m=1}^{n} $.  Note that $w_{j}=v_{j}$ for $j\neq i,i+1$, $w_{i}=v_{i+1}$ and $w_{i+1}=v_{i}.$
  \begin{figure}[h]
% \centering \includegraphics[scale=.75]{Capture2.PNG}
\def\svgwidth{300px}
\centering 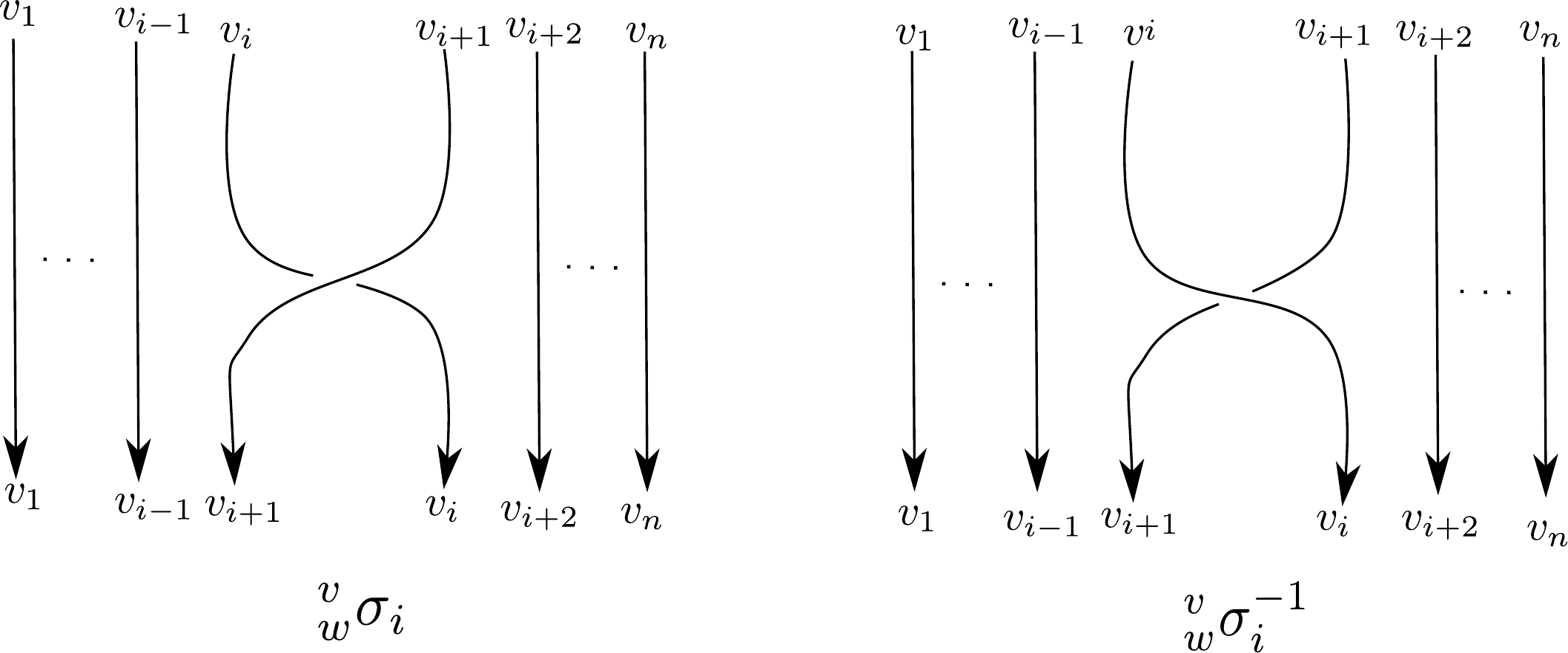
 \caption{Colored elementary braids}
 \label{figure:elementary braids }
 \end{figure}
 Any colored braid diagram can be drawn as a diagram whose crossings have distinct second coordinates. Therefore, any colored braid diagram can be obtained by stacking one colored elementary braid diagram below the other. So a braid diagram ${}^{v^{1}}_{v^{m+1}}\sigma$ can be written as the product of elementary braids as 
 ${}^{v^{1}}_{v^{2}}\sigma_{i_{1}}^{\pm}$ ${}^{v^{2}}_{v^{3}}\sigma_{i_{2}}^{\pm}\ldots {}^{v^{m}}_{v^{m+1}}\sigma_{i_{m}}^{\pm}$ where
  $v^{l}$ for $l\in \lbrace 1,2,\ldots,m+1\rbrace$ is the $n$-tuple indicating the color at the ends of the braid.
\par Now we define the category $\mathcal{D}_{n}^{k}$ of $R$-equivalent $k$-colored $n$-braid diagrams. The objects are the set of $n$-tuples $v=(v_{1},v_{2},\ldots,v_{n})$ with $v_{i}\in \lbrace 1,2,\ldots,k \rbrace$ and the set of morphisms $\textbf{Mor}(v,w)$  are the set of $R$-equivalent colored braid diagrams whose tuple of colors at the top and bottom ordered from left to right is $v$ and $w$ respectively. The composition of morphisms is a product of braid diagrams, and it's associative. The identity morphism of $v$ is the trivial braid diagram with coloring $v.$ 
 \par Similar to Artin's braid theory $\mathcal{D}_{n}^{k}$ satisfy the following relations \cite{kaul1994chern}. 
 \begin{equation*}
 \begin{split}
     {}^{v_{1}}_{v_{2}}\sigma_{i}\:{}^{v_{2}}_{v_{3}}\sigma_{i+1}\:{}^{v_{3}}_{v_{4}}\sigma_{i} &={}^{v_{1}}_{v_{2}}\sigma_{i+1}\:{}^{v_{2}}_{v_{3}}\sigma_{i}\:{}^{v_{3}}_{v_{4}}\sigma_{i+1}\\
   {}^{v_{1}}_{v_{2}}\sigma_{i}\:{}^{v_{2}}_{v_{3}}\sigma_{j} &={}^{v_{1}}_{v_{2}}\sigma_{j}\:{}^{v_{2}}_{v_{3}}\sigma_{i} 
   \end{split}
 \end{equation*}
The first and second equations follow from $3^{rd}$ colored braid Reidemeister move and diagram isotopy of braid diagrams, respectively. The natural covariant functor from $\mathcal{D}_{n}^{k}$ to $\mathcal{B}_{n}^{k}$ mapping $R$-equivalent braid diagram to the corresponding isotopic class of braids is an isomorphism.
 
 \begin{theorem}
 The category $\mathcal{D}_{n}^{k}$ is a groupoid.
 \end{theorem}
 \begin{proof}
To show $\mathcal{D}_{n}^{k}$ is a groupoid, it's enough to show that every morphism is an isomorphism in $\mathcal{D}_{n}^{k}$. Suppose ${}^{v^{1}}_{v^{m+1}}\sigma \in \textbf{Mor}(v^{1},v^{m+1}).$ Then by the explanation above ${}^{v^{1}}_{v^{m+1}}\sigma={}^{v^{1}}_{v^{2}}\sigma_{i_{1}}^{\pm}$  $ {}^{v^{2}}_{v^{3}}\sigma_{i_{2}}^{\pm}\ldots {}^{v^{m}}_{v^{m+1}}\sigma_{i_{m}}^{\pm}$ and it has an inverse of the form %$\left({}^{v^{1}}_{v^{m+1}}\sigma\right)^{-1}=$
${}^{v^{m+1}}_{v^{m}}\sigma_{i_{m}}^{\mp}\ldots{}^{v^{3}}_{v^{2}}\sigma_{i_{2}}^{\mp}$ 
 ${}^{v^{2}}_{v^{1}}\sigma_{i_{1}}^{\mp} \in \textbf{Mor}(v^{m+1},v^{1}) $. 
 \end{proof}
 Similarly the category $\mathcal{B}_{n}^{k}$ of $k$-colored $n$-braids is a groupoid.
 Now consider the set of isotopic classes of $k$-colored $n$-braids of the form ${}^v_v\sigma$, that is, those colored braids whose color at the top end of each string coincides with the same position at the bottom. This forms a subgroupoid of the groupoid of isotopic classes of $k$-colored $n$-braids. The $closure$ $of$ $a$ $colored$ $braid$ belonging to this subgroupoid is obtained by joining the top end of the string with the same position at the bottom of the braid. We denote the closure of the colored braid ${}^v_v\sigma$ by $\overline{{}^v_v\sigma}$ which is a colored link.
 \begin{theorem}
 Alexander's theorem for multicolored braids\cite{kaul1994chern}:\\ Given any $k$-colored link $L_{C}$ there exists a $k$-colored braid whose closure is $L_{C}$.
 \end{theorem}
 \begin{proof}
 Let $L$ be the underlying link obtained by neglecting the colors of $L_{C}.$ Then, by Alexander's theorem for classical braids, there exists a braid $\sigma$ whose closure is $L.$ By giving the strands of $\sigma$ the color of corresponding components of $L_{C}$, we get a $k$-colored braid of the form ${}^v_v\sigma$ whose closure is $L_{C}.$ 
 \end{proof}
\subsection*{Colored Markov Moves} \hfill \\
The colored Markov moves are defined as follows:\\
  1. Stabilisation\\
 The stabilisation move is illustrated in figure \ref{figure:stabilisation}. 
\begin{figure}[h]
% \centering \includegraphics[scale=.75]{Capture2.PNG}
\def\svgwidth{200px}
\centering 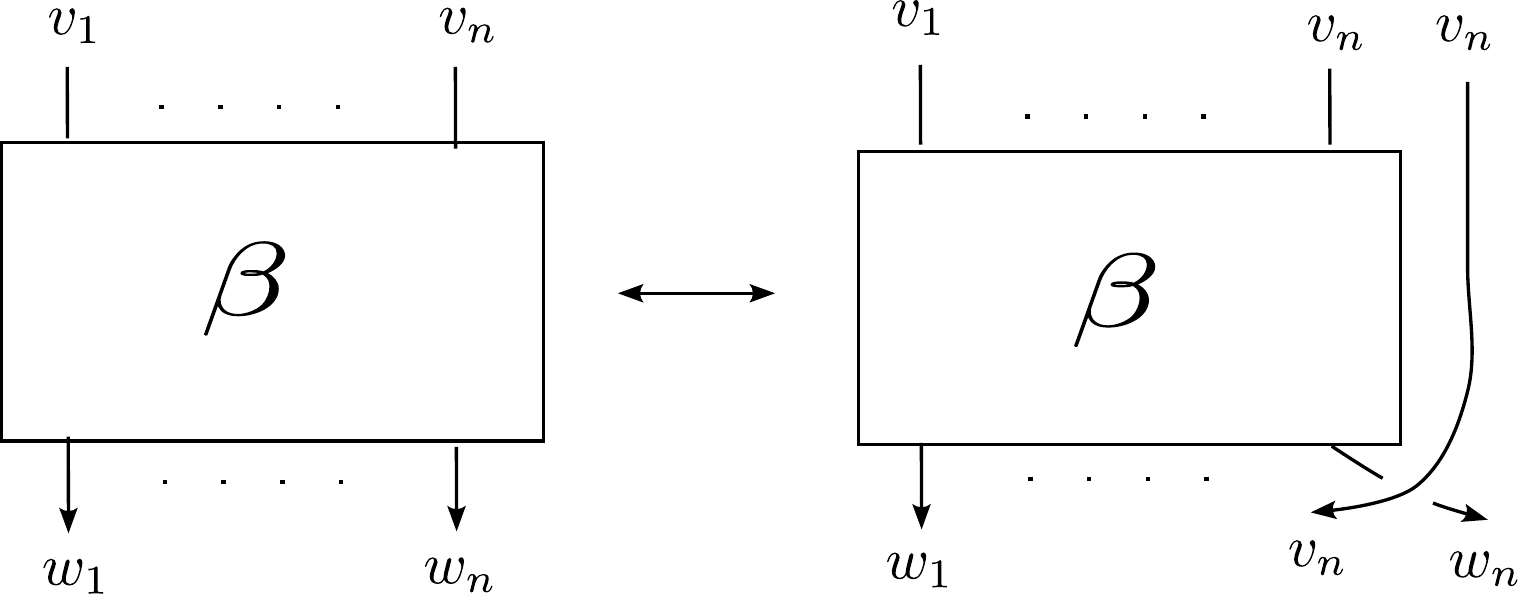
 \caption{stabilisation}
 \label{figure:stabilisation}
 \end{figure}
 The braid ${}^v_w\beta$ is an $n$-braid whereas the colored braid ${}^{v'}_{w'} \beta$  $ {}^{w'}_{w''}\sigma_{n}^{\pm}$ on the right hand side of the figure is an $n+1$ braid where $v'(n+1)=v(n), \\w'(n+1)=v(n), v'(i)=v(i)$ and $ w'(i)=w(i)$ for $i=1,\ldots,n $. Also $w''(n)=w'(n+1),w''(n+1)=w'(n)$ and $w''(j)=w'(j)$ for $j\neq n,n+1.$\\
 2. Conjugation\\
 The conjugation move transforms a colored braid ${}_{w}^{w}\beta$ into ${}_w^v\tau$  ${}_{w}^{w}\beta $  $\left( {}_w^v\tau \right)^{-1}$ and vice versa.
 The conjugation move for $\tau={}_w^v\sigma_{i}$ is illustrated in figure \ref{figure:conjugation}.
 \begin{figure}[h]
% \centering \includegraphics[scale=.75]{Capture2.PNG}
\def\svgwidth{150px}
\centering 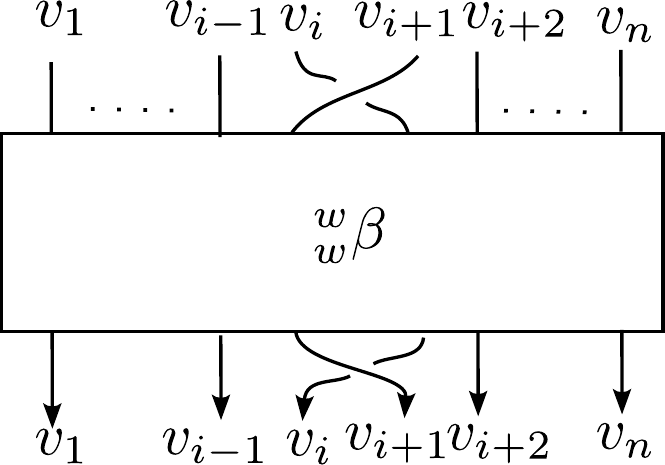
 \caption{Conjugation}
 \label{figure:conjugation}
 \end{figure}
 \begin{theorem}
     Markov Theorem for colored braids \cite{murakami1993state}: Let $\beta_{C}$ and $\beta'_{C'}$  be two colored braids whose colors at the top end coincide with the colors at the same position at the bottom. Then the closure of the colored braids $\beta_{C}$ and $\beta'_{C'}$ are the same if and only if there is a sequence of colored Markov moves connecting $\beta_{C}$ and $\beta'_{C'}$.
 \end{theorem}
\iffalse
 \begin{remark}
 Let $\beta_{C}$ and $\beta'_{C'}$  be two colored braids whose colors at the top end coincide with the same position at the bottom. If there is a sequence of colored Markov moves connecting $\beta_{C}$ and $\beta'_{C'}$, then the closure of the colored braids $\beta_{C}$ and $\beta'_{C'}$ are the same. Now, suppose the closure of colored braids $\overline{\beta_{C}}$ and $\overline{\beta'_{C'}}$ be isotopic. Let $\overline{\beta_{C}}=K_{C}$ and $\overline{\beta'_{C'}}=K'_{C'}$. Since $\overline{\beta_{C}}$ and $\overline{\beta_{C'}^{'}}$ are ambient isotopic as multicolored links the underlying braids $\beta$ and $\beta'$ which are obtained by neglecting the coloring of $\beta_{C}$ and $\beta'_{C'}$ are related by a sequence of Markov moves. By giving $\beta$ the color of $\beta_{C}$ and coloring the Markov moves by extending the coloring of the previous braid in the sequence, we get a sequence of colored braids connected by colored Markov moves starting at $\beta_{C}$ and ending at $\beta'_{C''}$. Here $C''$ may not necessarily be the coloring $C'$. But $\overline{\beta_{C}}=\overline{\beta_{C''}}=K_{C}$ and $\overline{\beta'_{C'}}=K'_{C'}$. Also $K_{C}$ and $K'_{C'}$ are isotopic as colored links.
 \end{remark}
 \fi
 \section{Topological multi-quandles}
 In \cite{rubinsztein2007topological}, topological quandles were used to construct classical link invariants. In this section, we 
 define topological multi-quandles and use them to associate a topological space with a colored link, which turns out to be a colored link invariant.
 \begin{definition}
     A $k$-quandle $\left(X,\triangleright_{i}\right)_{i=1}^{k}$ is said to be a topological $k$-quandle if $\triangleright_{i}^{\pm}$ is continous for $i=1,\ldots,k$.
 \end{definition}
 Loosely, we use the term topological multi-quandles if $k$ is not explicitly mentioned. Note that a topological 1-quandle is a topological quandle where a topological quandle is understood as defined in \cite{rubinsztein2007topological}.
 \begin{example}
 Let $\left( X,\triangleright \right)$ be a topological quandle. Then $\left( X,\triangleright_{i}\right)_{i=1}^{k}$ where $\triangleright_{i}=\triangleright^{\pm}$ for $i=1,\ldots,k$ is a topological $k$-quandle.
 \end{example}
 \begin{example}
 Let $\left( X,\triangleright \right)$ be a topological quandle. Define $x \triangleright_{i_{j}}y:=\left( \ldots \left( \left( x \triangleright y \right)\triangleright y \right) \ldots y \right)$ where $y$ is operated over $x,$ $i_{j}$ number of times. Then $\left( X,\triangleright_{i_{j}}\right)_{j=1}^{k}$ where $i_{j}\in \mathbb{N}$ is a topological $k$-quandle. 
 \end{example}
 \begin{example}
 The binary operation $x \triangleright_{t_{j}} y= t_{j} x+\left( 1-t_{j} \right) y$ for $j=1,\ldots ,k,$ $t_{j}\in \mathbb{R}\setminus \lbrace 0\rbrace$ gives a topological $k$-quandle structure $\left( \mathbb{R},\triangleright_{t_{j}}\right)_{j=1}^{k}$.
     
 \end{example}
 \textbf{Fundamental $k$-quandle of a $k$-colored link}:
 Let $D_{C}$ denote the diagram of a $k$-colored link $L_{C}$ with $n$ arcs, whose arcs are labeled by the set $Y=\lbrace x_{1},x_{2},\ldots,x_{n} \rbrace.$ By a $k-quandle$ $word$ in $Y$, we mean a string of elements of $Y$, $\triangleright_{j}$s and parentheses similar to words in group theory. Then the $fundamental$ $k-quandle$ of $L_{C}$ denoted by $F \left( D_{C} \right) $ is the set of equivalence classes of the $k$-quandle words in $Y$ modulo the crossing relations and $k$-quandle axioms.
 \iffalse
 The crossing relations for the positive and negative crossing of a $k$-colored link diagram is of the form $z=x \triangleright_{j} y$ and $z=x\triangleright_{j}^{-1}y$ respectively as illustrated in the figure \ref{figure:crossing}.
 \fi
 \begin{theorem}\label{thm:fundamental}
 The fundamental $k$-quandle is an invariant of\\ $k$-colored links.
 \end{theorem}
 \begin{proof}
 Consider figures \ref{figure:R1&R2} and \ref{figure:R3}.
 The condition that $x\triangleright_{i}x=x$ for a $k$-quandle implies that both the arcs in a twist in the first Reidemeister move should be labeled by the same element in  $F \left( D_{C} \right).$ Similarly, since $(x\triangleright_{i}y)\triangleright_{i}^{-1}y=x$, both the arcs at the end of the under strand that slides below the over strand in the second Reidemeister move should also be labeled by same element in  $F \left( D_{C} \right).$ By the second condition of Definition \ref{defn:k-quandle}, the labelings of the ending arc of the first strand in the third Reidemeister move before and after the move should also be the same. Therefore, the fundamental $k$-quandle of a $k$-colored link is invariant under colored Reidemeister moves. 
 \end{proof}
 \begin{definition}
     Let $X$ and $Y$ be (topological) $k$-quandles. A (continuous) map $f:X \rightarrow Y$ is said to be a (topological) $k$-quandle homomorphism if $f(x \triangleright_{i} y)=f(x) \triangleright_{i} f(y)$ for every $i=1,\ldots,k$
 \end{definition}
 Let $Q$ be a topological $k$-quandle, $L_{C}$ be a $k$-colored link and $D_{C}$ be its diagram. Let $\text{Hom}(F(D_{C}), Q)$ be the set of topological $k$-quandle homomorphisms from $F(D_{C})$ to $Q$ with compact open topology and $F(D_{C})$ having the discrete topology. 
 \begin{corollary}\label{cor:multi}
 $\text{Hom}(F(D_{C}),Q)$ is a $k$-colored link invariant.
 \end{corollary}
 \begin{proof}
 The corollary directly follows from theorem \ref{thm:fundamental}. 
 \end{proof}
 Let  ${}^v_w\sigma_{i}$ be a $k$-colored elementary $n$-braid where 
 $v=\left( v_{j}\right) $ and $w=\left( w_{j}\right) $ with 
$w_{k}=v_{k}$ for $k\neq i,i+1$, $w_{i}=v_{i+1}$ and $w_{i+1}=v_{i}.$ Then ${}^v_w\sigma_{i}$ induces a homeomorphism $${}^v_w\sigma_{i}:Q^{n} \rightarrow Q^{n}$$
%( x_{1},x_{2},\ldots,x_{n} ) \rightarrow  ( x_{1},\ldots,x_{i+1}, x_{i}\triangleright_{v_{i+1}} x_{i+1},\ldots,x_{n}) 
  \begin{equation*}
   ({}^v_w\sigma_{i}(x))_{j}=
  \begin{cases}
  x_{j} & \text{if}\: j \neq i,i+1\\
  x_{i+1} & \text{if}\: j=i\\
  x_{i}\triangleright_{v_{i+1}} x_{i+1} & \text{if} \: j = i+1
  \end{cases}   
 \end{equation*}

with $$\left({}^v_w\sigma_{i}\right)^{-1}={}^w_v\sigma_{i}^{-1}:Q^{n} \rightarrow Q^{n}$$
%$$\left( x_{1},x_{2},\ldots,x_{n}\right) \rightarrow \left( x_{1},\ldots,x_{i+1}\triangleright^{-1}_{v_{i}} x_{i},x_{i},\ldots,x_{n}\right)$$
 \begin{equation*}
   (({}^v_w\sigma_{i})^{-1}(x))_{j}=
   \begin{cases}
      x_{j} & \text{if } j \neq i, i+1 \\
      x_{i+1}\triangleright^{-1}_{v_{i}} x_{i} & \text{if } j=i \\
      x_{i} & \text{if } j = i+1
   \end{cases}   
\end{equation*}
Since any $k$-colored $n$-braid ${}^{v_{1}}_{v_{m+1}}\sigma$ is a product of colored elementary braids, it induces a homeomorphism from $Q^{n}$ to $Q^{n}$, which is a composition of homeomorphisms corresponding to colored elementary braids. The homeomorphisms in product notation should be composed from left to right. In other words if ${}^{v^{1}}_{v^{m+1}}\sigma={}^{v^{1}}_{v^{2}}\sigma_{i_{1}}^{\pm}$  $ {}^{v^{2}}_{v^{3}}\sigma_{i_{2}}^{\pm}\ldots {}^{v^{m}}_{v^{m+1}}\sigma_{i_{m}}^{\pm}$ then the homeomorphism is  ${}^{v^{m}}_{v^{m+1}}\sigma_{i_{m}}^{\pm} \circ \cdots \circ {}^{v^{2}}_{v^{3}}\sigma_{i_{2}}^{\pm} \circ {}^{v^{1}}_{v^{2}}\sigma_{i_{1}}^{\pm}$. We define $\hat{J}_{Q}(L_{C})$ to be the space of fixed points of the map ${}^v_v\sigma: Q^{n} \rightarrow Q^{n}$ where $L_{C}=\overline{{}^v_v\sigma}$. This construction is similar to that of $J_{Q}(L)$ in \cite{rubinsztein2007topological}. 
\begin{lemma}\label{lemma:homm}
    There exists a homeomorphism between $\text{Hom}(F(D_{C}),Q)$ and  $\hat{J}_{Q}(L_{C})$ .
\end{lemma}
\begin{proof}
 Define a map $\phi:\text{Hom}(F(\overline{{}^v_v\sigma}),Q) \rightarrow \hat{J}_{Q}(L_{C})$ as $$\phi(f):=\left( f(y_{1}),f(y_{2}),\ldots,f(y_{n})\right)$$ where $y_{1},y_{2},\ldots,y_{n}$ denotes the labels of top arcs of the $n$-braid ${}^v_v\sigma$ from left to right. To show the map $\phi$ is continuous it's enough to show that the $i^{th}$ coordinate function
 $$ f \xrightarrow{\phi_{i}}f(y_{i})$$ is continuous. But if $f(y_{i}) \in U$ then $\phi_i^{-1}(U)=\mathcal{B}(\lbrace y_{i} \rbrace ,U)$, where 
 $\mathcal{B}(\lbrace y_{i} \rbrace ,U)$ is a sub basis element of $\text{Hom}(F(\overline{{}^v_v\sigma}),Q)$, whose topology is compact open topology.
 Now define $$\phi^{-1}:\hat{J}_{Q}(L_{C}) \rightarrow \text{Hom}(F(\overline{{}^v_v\sigma}),Q)$$ as
 $$(a_{1},a_{2},\ldots,a_{n}) \xrightarrow{\phi^{-1}} g $$ where $g(y_{i}):=a_{i}$ for $i=1,2,\ldots,n.$ As $(a_{1},a_{2},\ldots,a_{n}) \in \hat{J}_{Q}(L_{C})$, $g$ could be defined on all arcs of $\overline{{}^v_v\sigma}$ and extended over $F(\overline{{}^v_v\sigma}).$ 
 Now any compact set $K$ of $F(\overline{{}^v_v\sigma})$ will be of the form $\lbrace  w_{1},w_{2},\ldots,w_{m} \rbrace$ where each $w_{i}$ is a $k$-quandle word in $F(\overline{{}^v_v\sigma})$ with generators $y_{1},y_{2},\ldots,y_{n}$. Also, $w_{i}$ can be seen as a function from $Q^{n}$ to $Q$ which maps $(a_{1},a_{2},\ldots, a_{n})$ to $w_{i}(a_{1},a_{2},\ldots, a_{n})$ where $w_{i}(a_{1},a_{2},\ldots, a_{n})$ is the element in quandle $Q$ obtained by substituting $a_{j}$s instead of $y_{j}$s in $w_{i}.$ The function $w_{i}$ is continuous as the quandle operations are continuous.
 Note that
 \iffalse
 \begin{equation*}
 \begin{split}
 \mathcal{B}(K,U) &=\lbrace f \in \text{Hom}(F(\overline{{}^v_v\sigma}),Q) \mid f(w_{i}) \in U \; \forall i \in  \lbrace1,2,\ldots,m \rbrace \rbrace\\
                  &= \cap_{i \in  \lbrace 1,2,\ldots,m\rbrace}\lbrace f \in \text{Hom}(F(\overline{{}^v_v\sigma}),Q)\mid f(w_{i}) \in U \rbrace 
 \end{split}
 \end{equation*}
 \fi
 \begin{equation*}
\begin{split}
\mathcal{B}(K,U) &= \left\lbrace f \in \text{Hom}(F(\overline{{}^v_v\sigma}),Q) \;\middle|\; f(w_{i}) \in U, \;\forall i \in \{1,2,\ldots,m\} \right\rbrace \\
&= \bigcap_{i \in \{1,2,\ldots,m\}}\left\lbrace f \in \text{Hom}(F(\overline{{}^v_v\sigma}),Q) \;\middle|\; f(w_{i}) \in U \right\rbrace
\end{split}
\end{equation*}

 Then
 \iffalse
 \begin{equation*}
  \begin{split}
   \left( \phi^{-1} \right)^{-1}(\mathcal{B}(K,U)) 
  &= \lbrace (a_{1},a_{2},\ldots ,a_{n}) \in  \hat{J}_{Q}(L_{C}) \mid \\w_{i}(a_{1},a_{2},\ldots, a_{n}) \in U  
   \forall i \in \lbrace 1,2,\ldots,m \rbrace \rbrace   \\
  &=\cap_{i=1}^{m}\lbrace (a_{1},a_{2},\ldots ,a_{n}) \in  \hat{J}_{Q}(L_{C}) \mid \\w_{i}(a_{1},a_{2},\ldots, a_{n}) \in U \rbrace\\
  &= \cap_{i=1}^{m} w_{i}^{-1}(U) \cap \hat{J}_{Q}(L_{C})
  \end{split}   
 \end{equation*}
 \begin{equation*}
\begin{split}
&\left( \phi^{-1} \right)^{-1}(\mathcal{B}(K,U)) \\
&= \left\lbrace (a_{1},a_{2},\ldots ,a_{n}) \in  \hat{J}_{Q}(L_{C}) \, \middle| \, w_{i}(a_{1},a_{2},\ldots, a_{n}) \in U, \forall i \in \{1,2,\ldots,m\} \right\rbrace \\
&= \bigcap_{i=1}^{m}\left\lbrace (a_{1},a_{2},\ldots ,a_{n}) \in  \hat{J}_{Q}(L_{C}) \, \middle| \, w_{i}(a_{1},a_{2},\ldots, a_{n}) \in U \right\rbrace \\
&= \bigcap_{i=1}^{m} w_{i}^{-1}(U) \cap \hat{J}_{Q}(L_{C})
\end{split}
\end{equation*}
\fi
\begin{equation*}
\begin{split}
&\left( \phi^{-1} \right)^{-1}(\mathcal{B}(K,U)) \\
&= \left\lbrace (a_{1},a_{2},\ldots ,a_{n}) \in  \hat{J}_{Q}(L_{C}) \, \middle| \, \begin{array}{@{}l@{}}w_{i}(a_{1},a_{2},\ldots, a_{n}) \in U, \\ \forall i \in \{1,2,\ldots,m\}\end{array} \right\rbrace \\
&= \bigcap_{i=1}^{m}\left\lbrace (a_{1},a_{2},\ldots ,a_{n}) \in  \hat{J}_{Q}(L_{C}) \, \middle| \, w_{i}(a_{1},a_{2},\ldots, a_{n}) \in U \right\rbrace \\
&= \bigcap_{i=1}^{m} w_{i}^{-1}(U) \cap \hat{J}_{Q}(L_{C})
\end{split}
\end{equation*}
 Therefore $\phi^{-1}$ is continuous as $\left( \phi^{-1} \right)^{-1}( \mathcal{B}(K,U)) $ is an open set.
 Note that $\phi \circ \phi^{-1}=I_{\hat{J}_{Q}(L_{C})}$ and $\phi^{-1} \circ \phi=I_{\text{Hom}(F(\overline{{}^v_v\sigma}),Q)}.$ Therefore $\phi$ is a homeomorphism.
\end{proof}
The following theorems are analogues of theorems in \cite{rubinsztein2007topological}.
\begin{theorem}\label{thm:invariant}
 Given a topological multi-quandle $Q$, the space of fixed points $\hat{J}_{Q}(L_{C})$ is a multicolored link invariant.    
\end{theorem}
\begin{proof}
    We want to prove that $\hat{J}_{Q}(L_{C})$ is invariant under the colored Markov moves.\\
$(1)$ The maps ${}_{v}^{v}\sigma$ and ${}_v^w\tau$ ${}_{v}^{v}\sigma $ $\left( {}_v^w\tau \right)^{-1}$ have homeomorphic space of fix points as conjugate homeomorphisms have homeomorphic fixed point spaces. \\
$(2)$ The maps ${}^{v'}_{v'} \beta $ ${}^{v'}_{w'}\sigma_{n}^{\pm}$  and  ${}^v_v \beta$ have homeomorphic space of fix points where $v'(n+1)=v(n), w'(n+1)=v'(n), w'(n)=v'(n+1)$ and \\$v'(i)=v(i)$ for $i=1,\ldots, n$ and $w'(i)=v'(i)$ for $i=1,\ldots, n-1$.\\
If $(x_{1},x_{2},\ldots,x_{n}) \in Q^{n}$ is a fixed point of ${}^v_v \beta$ then \\$(x_{1},x_{2},\ldots,x_{n},x_{n}) \in Q^{n+1}$ is a fixed point of ${}^{v'}_{v'} \beta \;{}^{v'}_{w'}\sigma_{n}^{\pm}$.\\
Conversely suppose $(x_{1},x_{2},\ldots,x_{n+1}) \in Q^{n+1}$ is a fixed point of $ {}^{v'}_{v'} \beta \; {}^{v'}_{w'}\sigma_{n}$ then
 \begin{equation*}
    \begin{split}
     {}^{v'}_{v'}\beta\;{}^{v'}_{w'}\sigma_{n} (x_{1},x_{2},\ldots,x_{n},x_{n+1}) = {}^{v'}_{w'}\sigma_{n}\;({}^{v'}_{v'} \beta\;(x_{1},x_{2},\ldots,x_{n}),x_{n+1})\\
     \end{split}
 \end{equation*}
 Suppose ${}^{v'}_{v'} \beta\;(x_{1},x_{2},\ldots,x_{n})=(y_{1},y_{2},\ldots,y_{n}),$ then
 \begin{equation*}
    \begin{split}
{}^{v'}_{v'}\beta \; {}^{v'}_{w'}\sigma_{n} (x_{1},x_{2},\ldots,x_{n},x_{n+1})   &={}^{v'}_{w'}\sigma_{n}(y_{1},y_{2},\ldots,y_{n},x_{n+1})\\
    &=(y_{1},y_{2},\ldots,y_{n-1},x_{n+1},y_{n}\triangleright_{v_{n}} x_{n+1})\\
    &=(x_{1},x_{2},\ldots,x_{n-1},x_{n},x_{n+1})
    \end{split}
    \end{equation*}
   
It follows that 
\begin{equation*}
\begin{split}
y_{i} &= x_{i} \; \text{for} \; i=1,2,\ldots ,n-1 \\
x_{n+1} &= x_{n}\\
y_{n} &= x_{n+1}
\end{split}
\end{equation*}
Therefore $(x_{1},x_{2},\ldots,x_{n})\in Q^{n}$ is a fixed point of ${}^{v}_{v} \beta$ if and only if $(x_{1},x_{2},\ldots,x_{n},x_{n})\in Q^{n+1}$ is a fixed point of $ {}^{v'}_{v'} \beta \; {}^{v'}_{w'}\sigma_{n}$.
\end{proof}
\begin{remark}
For a diagram $D_{C}$ of a $k$-colored link $L_{C}$ with $n$ arcs, we define the $\textit{space of colorings of}$ $D_{C}$ by a topological $k$-quandle $Q$ to be set of colorings of $D_{C}$ whose elements are $n$-tuples in $Q$ equipped with a subspace topology of $Q^{n}.$ We denote this space as $S_{Q}(D_{C})$. There exists a homeomorphism between $S_{Q}(D_{C})$ and $\text{Hom}(F(D_{C}),Q)$. The reasoning is similar to the proof of lemma \ref{lemma:homm}. Let $\beta_{C}$ be a $k$-colored braid representing $L_{C}$. Then $S_{Q}(D_{C})=\text{Hom}(F(D_{C}),Q)=\text{Hom}(F(\overline{\beta}_{C}),Q)=\hat{J}_{Q}(L_{C}).$
\end{remark}
\iffalse    $(3)$ Let $\beta_{C}$ and $\beta_{C'}$ be two colored braids related by equivalent coloring. Then by corollary \ref{cor:multi}, $\text{Hom}(\overline{{}^v_w\beta_{C}},Q)$ $=\text{Hom}(\overline{{}^v_w\beta_{C'}},Q)$. Therefore by lemma \ref{lemma:homm}, $\hat{J}_{Q}(L_{C})$ remains unchanged for $\beta_{C}$ and $\beta_{C'}.$
\fi
\begin{theorem} \label{thm:disjoint links}
    The invariant space $\hat{J}_Q(L^1_C \amalg L^2_{C'})$ of  $L^1_C \amalg L^2_{C'}$, which is the disjoint sum of $k$-colored links $L^1_C$ and $L^2_{C'}$ is homeomorphic to $\hat{J}_Q(L^1_C) \times \hat{J}_Q(L^2_{C'})$, which is the product of invariant spaces of $L^1_C$ and $L^2_{C'}$.

\end{theorem}
\begin{proof}
Suppose $\gamma_{1}$ and $\gamma_{2}$ be two $k$-colored braids with $m$ and $n$ strands respectively, such that $\overline{\gamma_{1}}=L^{1}_{C}$ and $\overline{\gamma_{2}}=L^{2}_{C'}.$ Then $\hat{J}_{Q}(L^{1}_{C} \amalg L^{2}_{C'})$ is the space of fixed points of the map 

\begin{equation*}
\begin{aligned}
\gamma_{1}\gamma_{2} & : Q^{m+n} \rightarrow Q^{m+n} \\
(x_{1},\ldots,x_{n},x_{n+1},\ldots,x_{n+m}) & \mapsto (\gamma_{1}(x_{1},x_{2},\ldots,x_{n}),\gamma_{2}(x_{n+1},\ldots,x_{n+m})).
\end{aligned}
\end{equation*}

    Therefore $(x_{1},x_{2},\ldots,x_{n},x_{n+1},\ldots,x_{n+m})$ is a fixed point of $\gamma_{1}\gamma_{2}$ if and only if $(x_{1},x_{2},\ldots,x_{n})$ is a fixed point of $\gamma_{1}$ and $(x_{n+1},\ldots,x_{n+m})$ is a fixed point of $\gamma_{2}$. Now, the theorem follows.
\end{proof}
\subsection*{A dichromatic link invariant}
In \cite{lee2021diquandles}, it's shown that the cardinality of colorings of dichromatic links by a diquandle is a dichromatic link invariant. It's not difficult to see that for a dichromatic link with diagram $D_C$, $\vert S_{Q}(D_{C}) \vert$ which is the cardinality of space of colorings of $D_{C}$ by a topological $2$-quandle $Q$ is the aforementioned dichromatic link invariant. We use the term $topological$ $diquandle$ hereinafter instead of topological $2$-quandle as a topological $2$-quandle $X$ is a diquandle with a topology such that both the quandle operations and their inverses are continuous. 
\begin{example}
    Let $G$ be a cyclic topological group. Suppose $\sigma_{1}$ and $\sigma_{2}$ be continous automorphisms of $G$. Then, the quandle operations 
    $$ x\triangleright_{i}y=\sigma_{i}\left( y \right)^{-1}\sigma_{i}(x)y \; \text{for } i=1,2 $$
    makes $G$ a topological diquandle. The verification is a direct check keeping in mind that the group of automorphisms of $G$ is abelian.
\end{example}
\begin{example}
Consider the following two dichromatic links in figure \ref{figure:link}.
\begin{figure}[h]
% \centering \includegraphics[scale=.75]{Capture2.PNG}
\def\svgwidth{200px}
\centering 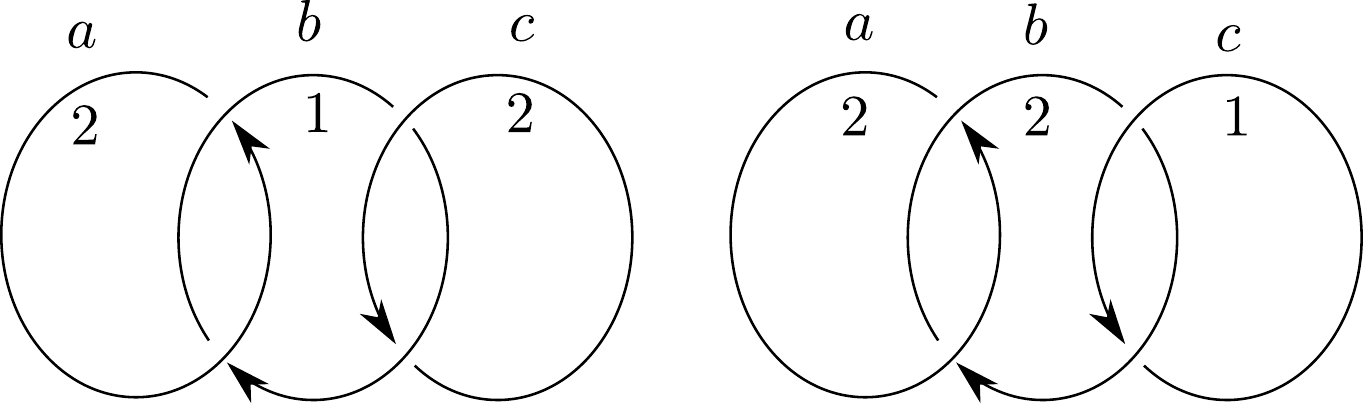
 \caption{Link with two colors}
 \label{figure:link}
 \end{figure}
 Let the topological diquandle be the unit sphere $S^{1}$ with $x\triangleright_{1} y=x$ and $x\triangleright_{2}y=2\langle x,y \rangle y -x$.\\
 The first figure satisfies the equations
 \begin{equation*}
 \begin{split}
     a \triangleright_{1} b=a \\
     c \triangleright_{1} b=c \\
     \left( b \triangleright_{2} c \right)\triangleright_{2} a =b
\end{split}
 \end{equation*}
 which will imply $c=\pm a.$ Therefore $\hat{J}_{Q}(L_{1})=\lbrace(a,b,c) \in S^{1} \times S^{1} \times S^{1} \mid c= \pm a \rbrace= T^{1} \sqcup T^{1}$.
The second dichromatic diagram satisfies the equations
 \begin{equation*}
 \begin{split}
     a \triangleright_{2} b=a \\
     c \triangleright_{2} b=c \\
     \left( b \triangleright_{1} c \right)\triangleright_{2} a =b
\end{split}
 \end{equation*}
 which will imply $a=\pm b$ and $c=\pm b.$ Therefore $\hat{J}_{Q}(L_{2})=\lbrace(a,b,c) \in S^{1} \times S^{1} \times S^{1} \mid a= \pm b,\; c= \pm b \rbrace= S^{1} \sqcup S^{1}\sqcup S^{1}\sqcup S^{1}.$
\end{example}

%For acknowledgements section, please don't number the section, please begin it with \section*{Acknowledgements}
\section*{Acknowledgments} We would like to express our gratitude to Haritha N for her valuable contributions and insightful discussions on the topic.

% You may incorporate your references as follows in your main tex file.
% Using BibTex is not recommended but can be handled.

\bibliography{ref}{}
\bibliographystyle{plain}
\medskip
% The data information below will be filled by AIMS editorial staff
%Received xxxx 20xx; revised xxxx 20xx.
\medskip

\end{document}

%% file: fig2.pdf_tex
%% Creator: Inkscape 1.1.2 (b8e25be833, 2022-02-05), www.inkscape.org
%% Accompanies image file 'fig2.pdf' (pdf, eps, ps)
\begingroup%
  \makeatletter%
  \providecommand\color[2][]{%
    \errmessage{(Inkscape) Color is used for the text in Inkscape, but the package 'color.sty' is not loaded}%
    \renewcommand\color[2][]{}%
  }%
  \providecommand\transparent[1]{%
    \errmessage{(Inkscape) Transparency is used (non-zero) for the text in Inkscape, but the package 'transparent.sty' is not loaded}%
    \renewcommand\transparent[1]{}%
  }%
  \providecommand\rotatebox[2]{#2}%
  \newcommand*\fsize{\dimexpr\f@size pt\relax}%
  \newcommand*\lineheight[1]{\fontsize{\fsize}{#1\fsize}\selectfont}%
  \ifx\svgwidth\undefined%
    \setlength{\unitlength}{1035.36640063bp}%
    \ifx\svgscale\undefined%
      \relax%
    \else%
      \setlength{\unitlength}{\unitlength * \real{\svgscale}}%
    \fi%
  \else%
    \setlength{\unitlength}{\svgwidth}%
  \fi%
  \global\let\svgwidth\undefined%
  \global\let\svgscale\undefined%
  \makeatother%
  \begin{picture}(1,0.64925757)%
    \lineheight{1}%
    \setlength\tabcolsep{0pt}%
    \put(0,0){\includegraphics[width=\unitlength,page=1]{fig2.pdf}}%
  \end{picture}%
\endgroup%

%% file: fig3.pdf_tex
%% Creator: Inkscape 1.1.2 (b8e25be833, 2022-02-05), www.inkscape.org
%% Accompanies image file 'fig3.pdf' (pdf, eps, ps)
\begingroup%
  \makeatletter%
  \providecommand\color[2][]{%
    \errmessage{(Inkscape) Color is used for the text in Inkscape, but the package 'color.sty' is not loaded}%
    \renewcommand\color[2][]{}%
  }%
  \providecommand\transparent[1]{%
    \errmessage{(Inkscape) Transparency is used (non-zero) for the text in Inkscape, but the package 'transparent.sty' is not loaded}%
    \renewcommand\transparent[1]{}%
  }%
  \providecommand\rotatebox[2]{#2}%
  \newcommand*\fsize{\dimexpr\f@size pt\relax}%
  \newcommand*\lineheight[1]{\fontsize{\fsize}{#1\fsize}\selectfont}%
  \ifx\svgwidth\undefined%
    \setlength{\unitlength}{682.9602057bp}%
    \ifx\svgscale\undefined%
      \relax%
    \else%
      \setlength{\unitlength}{\unitlength * \real{\svgscale}}%
    \fi%
  \else%
    \setlength{\unitlength}{\svgwidth}%
  \fi%
  \global\let\svgwidth\undefined%
  \global\let\svgscale\undefined%
  \makeatother%
  \begin{picture}(1,0.66332618)%
    \lineheight{1}%
    \setlength\tabcolsep{0pt}%
    \put(0,0){\includegraphics[width=\unitlength,page=1]{fig3.pdf}}%
  \end{picture}%
\endgroup%

%% file: fig1.pdf_tex
%% Creator: Inkscape 1.1.2 (b8e25be833, 2022-02-05), www.inkscape.org
%% Accompanies image file 'fig1.pdf' (pdf, eps, ps)
\begingroup%
  \makeatletter%
  \providecommand\color[2][]{%
    \errmessage{(Inkscape) Color is used for the text in Inkscape, but the package 'color.sty' is not loaded}%
    \renewcommand\color[2][]{}%
  }%
  \providecommand\transparent[1]{%
    \errmessage{(Inkscape) Transparency is used (non-zero) for the text in Inkscape, but the package 'transparent.sty' is not loaded}%
    \renewcommand\transparent[1]{}%
  }%
  \providecommand\rotatebox[2]{#2}%
  \newcommand*\fsize{\dimexpr\f@size pt\relax}%
  \newcommand*\lineheight[1]{\fontsize{\fsize}{#1\fsize}\selectfont}%
  \ifx\svgwidth\undefined%
    \setlength{\unitlength}{376.5002485bp}%
    \ifx\svgscale\undefined%
      \relax%
    \else%
      \setlength{\unitlength}{\unitlength * \real{\svgscale}}%
    \fi%
  \else%
    \setlength{\unitlength}{\svgwidth}%
  \fi%
  \global\let\svgwidth\undefined%
  \global\let\svgscale\undefined%
  \makeatother%
  \begin{picture}(1,0.50154201)%
    \lineheight{1}%
    \setlength\tabcolsep{0pt}%
    \put(0,0){\includegraphics[width=\unitlength,page=1]{fig1.pdf}}%
  \end{picture}%
\endgroup%

%% file: fig6.pdf_tex
%% Creator: Inkscape 1.1.2 (b8e25be833, 2022-02-05), www.inkscape.org
%% Accompanies image file 'fig6.pdf' (pdf, eps, ps)
\begingroup%
  \makeatletter%
  \providecommand\color[2][]{%
    \errmessage{(Inkscape) Color is used for the text in Inkscape, but the package 'color.sty' is not loaded}%
    \renewcommand\color[2][]{}%
  }%
  \providecommand\transparent[1]{%
    \errmessage{(Inkscape) Transparency is used (non-zero) for the text in Inkscape, but the package 'transparent.sty' is not loaded}%
    \renewcommand\transparent[1]{}%
  }%
  \providecommand\rotatebox[2]{#2}%
  \newcommand*\fsize{\dimexpr\f@size pt\relax}%
  \newcommand*\lineheight[1]{\fontsize{\fsize}{#1\fsize}\selectfont}%
  \ifx\svgwidth\undefined%
    \setlength{\unitlength}{607.23448438bp}%
    \ifx\svgscale\undefined%
      \relax%
    \else%
      \setlength{\unitlength}{\unitlength * \real{\svgscale}}%
    \fi%
  \else%
    \setlength{\unitlength}{\svgwidth}%
  \fi%
  \global\let\svgwidth\undefined%
  \global\let\svgscale\undefined%
  \makeatother%
  \begin{picture}(1,0.76915101)%
    \lineheight{1}%
    \setlength\tabcolsep{0pt}%
    \put(0,0){\includegraphics[width=\unitlength,page=1]{fig6.pdf}}%
  \end{picture}%
\endgroup%

%% file: fig4.pdf_tex
%% Creator: Inkscape 1.1.2 (b8e25be833, 2022-02-05), www.inkscape.org
%% Accompanies image file 'fig4.pdf' (pdf, eps, ps)
\begingroup%
  \makeatletter%
  \providecommand\color[2][]{%
    \errmessage{(Inkscape) Color is used for the text in Inkscape, but the package 'color.sty' is not loaded}%
    \renewcommand\color[2][]{}%
  }%
  \providecommand\transparent[1]{%
    \errmessage{(Inkscape) Transparency is used (non-zero) for the text in Inkscape, but the package 'transparent.sty' is not loaded}%
    \renewcommand\transparent[1]{}%
  }%
  \providecommand\rotatebox[2]{#2}%
  \newcommand*\fsize{\dimexpr\f@size pt\relax}%
  \newcommand*\lineheight[1]{\fontsize{\fsize}{#1\fsize}\selectfont}%
  \ifx\svgwidth\undefined%
    \setlength{\unitlength}{904.16236535bp}%
    \ifx\svgscale\undefined%
      \relax%
    \else%
      \setlength{\unitlength}{\unitlength * \real{\svgscale}}%
    \fi%
  \else%
    \setlength{\unitlength}{\svgwidth}%
  \fi%
  \global\let\svgwidth\undefined%
  \global\let\svgscale\undefined%
  \makeatother%
  \begin{picture}(1,0.75236491)%
    \lineheight{1}%
    \setlength\tabcolsep{0pt}%
    \put(0,0){\includegraphics[width=\unitlength,page=1]{fig4.pdf}}%
  \end{picture}%
\endgroup%

%% file: fig5.pdf_tex
%% Creator: Inkscape 1.1.2 (b8e25be833, 2022-02-05), www.inkscape.org
%% Accompanies image file 'fig5.pdf' (pdf, eps, ps)
\begingroup%
  \makeatletter%
  \providecommand\color[2][]{%
    \errmessage{(Inkscape) Color is used for the text in Inkscape, but the package 'color.sty' is not loaded}%
    \renewcommand\color[2][]{}%
  }%
  \providecommand\transparent[1]{%
    \errmessage{(Inkscape) Transparency is used (non-zero) for the text in Inkscape, but the package 'transparent.sty' is not loaded}%
    \renewcommand\transparent[1]{}%
  }%
  \providecommand\rotatebox[2]{#2}%
  \newcommand*\fsize{\dimexpr\f@size pt\relax}%
  \newcommand*\lineheight[1]{\fontsize{\fsize}{#1\fsize}\selectfont}%
  \ifx\svgwidth\undefined%
    \setlength{\unitlength}{1014.78695457bp}%
    \ifx\svgscale\undefined%
      \relax%
    \else%
      \setlength{\unitlength}{\unitlength * \real{\svgscale}}%
    \fi%
  \else%
    \setlength{\unitlength}{\svgwidth}%
  \fi%
  \global\let\svgwidth\undefined%
  \global\let\svgscale\undefined%
  \makeatother%
  \begin{picture}(1,0.41646231)%
    \lineheight{1}%
    \setlength\tabcolsep{0pt}%
    \put(0,0){\includegraphics[width=\unitlength,page=1]{fig5.pdf}}%
  \end{picture}%
\endgroup%

%% file: fig7.pdf_tex
%% Creator: Inkscape 1.1.2 (b8e25be833, 2022-02-05), www.inkscape.org
%% Accompanies image file 'fig7.pdf' (pdf, eps, ps)
\begingroup%
  \makeatletter%
  \providecommand\color[2][]{%
    \errmessage{(Inkscape) Color is used for the text in Inkscape, but the package 'color.sty' is not loaded}%
    \renewcommand\color[2][]{}%
  }%
  \providecommand\transparent[1]{%
    \errmessage{(Inkscape) Transparency is used (non-zero) for the text in Inkscape, but the package 'transparent.sty' is not loaded}%
    \renewcommand\transparent[1]{}%
  }%
  \providecommand\rotatebox[2]{#2}%
  \newcommand*\fsize{\dimexpr\f@size pt\relax}%
  \newcommand*\lineheight[1]{\fontsize{\fsize}{#1\fsize}\selectfont}%
  \ifx\svgwidth\undefined%
    \setlength{\unitlength}{728.366138bp}%
    \ifx\svgscale\undefined%
      \relax%
    \else%
      \setlength{\unitlength}{\unitlength * \real{\svgscale}}%
    \fi%
  \else%
    \setlength{\unitlength}{\svgwidth}%
  \fi%
  \global\let\svgwidth\undefined%
  \global\let\svgscale\undefined%
  \makeatother%
  \begin{picture}(1,0.3911781)%
    \lineheight{1}%
    \setlength\tabcolsep{0pt}%
    \put(0,0){\includegraphics[width=\unitlength,page=1]{fig7.pdf}}%
  \end{picture}%
\endgroup%

%% file: fig8.pdf_tex
%% Creator: Inkscape 1.1.2 (b8e25be833, 2022-02-05), www.inkscape.org
%% Accompanies image file 'fig8.pdf' (pdf, eps, ps)
\begingroup%
  \makeatletter%
  \providecommand\color[2][]{%
    \errmessage{(Inkscape) Color is used for the text in Inkscape, but the package 'color.sty' is not loaded}%
    \renewcommand\color[2][]{}%
  }%
  \providecommand\transparent[1]{%
    \errmessage{(Inkscape) Transparency is used (non-zero) for the text in Inkscape, but the package 'transparent.sty' is not loaded}%
    \renewcommand\transparent[1]{}%
  }%
  \providecommand\rotatebox[2]{#2}%
  \newcommand*\fsize{\dimexpr\f@size pt\relax}%
  \newcommand*\lineheight[1]{\fontsize{\fsize}{#1\fsize}\selectfont}%
  \ifx\svgwidth\undefined%
    \setlength{\unitlength}{319.05427255bp}%
    \ifx\svgscale\undefined%
      \relax%
    \else%
      \setlength{\unitlength}{\unitlength * \real{\svgscale}}%
    \fi%
  \else%
    \setlength{\unitlength}{\svgwidth}%
  \fi%
  \global\let\svgwidth\undefined%
  \global\let\svgscale\undefined%
  \makeatother%
  \begin{picture}(1,0.69894646)%
    \lineheight{1}%
    \setlength\tabcolsep{0pt}%
    \put(0,0){\includegraphics[width=\unitlength,page=1]{fig8.pdf}}%
  \end{picture}%
\endgroup%

%% file: fig10.pdf_tex
%% Creator: Inkscape 1.1.2 (b8e25be833, 2022-02-05), www.inkscape.org
%% Accompanies image file 'fig10.pdf' (pdf, eps, ps)
\begingroup%
  \makeatletter%
  \providecommand\color[2][]{%
    \errmessage{(Inkscape) Color is used for the text in Inkscape, but the package 'color.sty' is not loaded}%
    \renewcommand\color[2][]{}%
  }%
  \providecommand\transparent[1]{%
    \errmessage{(Inkscape) Transparency is used (non-zero) for the text in Inkscape, but the package 'transparent.sty' is not loaded}%
    \renewcommand\transparent[1]{}%
  }%
  \providecommand\rotatebox[2]{#2}%
  \newcommand*\fsize{\dimexpr\f@size pt\relax}%
  \newcommand*\lineheight[1]{\fontsize{\fsize}{#1\fsize}\selectfont}%
  \ifx\svgwidth\undefined%
    \setlength{\unitlength}{654.20095178bp}%
    \ifx\svgscale\undefined%
      \relax%
    \else%
      \setlength{\unitlength}{\unitlength * \real{\svgscale}}%
    \fi%
  \else%
    \setlength{\unitlength}{\svgwidth}%
  \fi%
  \global\let\svgwidth\undefined%
  \global\let\svgscale\undefined%
  \makeatother%
  \begin{picture}(1,0.29576504)%
    \lineheight{1}%
    \setlength\tabcolsep{0pt}%
    \put(0,0){\includegraphics[width=\unitlength,page=1]{fig10.pdf}}%
  \end{picture}%
\endgroup%